\documentclass{amsart}

\usepackage{graphicx}%
\usepackage{amsmath,amssymb,amsfonts}%
\usepackage{amsthm}%
\usepackage{physics}
\usepackage{enumitem}
\usepackage{array}

\usepackage[colorlinks=true,allcolors=blue]{hyperref}

\newcommand{\Pstab}{\mathcal{P}^{\textup{stab}}}
\newcommand{\Punst}{\mathcal{P}^{\textup{unst}}}
\newcommand{\Pbif}{\mathcal{P}^{\textup{bif}}}

\newcommand{\inner}[2]{\langle #1, #2\rangle}

\newcommand{\phitil}{\widetilde{\phi}}

\newcommand{\ftil}{\widetilde{f}}
\newcommand{\Xtil}{\widetilde{X}}
\newcommand{\Wunst}{W^{\textup{u}}}
\newcommand{\Wstab}{W^{\textup{s}}}

\newcommand{\IA}{I_{\mathcal{P}_A}}

\DeclareMathOperator{\sgn}{sgn}

\DeclareMathOperator{\enc}{enc}

\DeclareMathOperator{\dist}{dist}

\newtheorem{lemma}{Lemma}

\newtheorem{definition}{Definition}

\newtheorem{remark}{Remark}

\newtheorem{manualtheoreminner}{Theorem}
\newenvironment{manualtheorem}[1]{%
  \renewcommand\themanualtheoreminner{#1}%
  \manualtheoreminner
}{\endmanualtheoreminner}

\begin{document}

\title{Controllability in a class of cancer therapy models with co-evolving resistance} 

\author{Frederik J. Thomsen}
\address{Delft Institute of Applied Mathematics, Delft University of Technology, Mekelweg 4, Delft 2628DC, Netherlands}
\email{f.j.thomsen@tudelft.nl}

\author{Johan L. A. Dubbeldam}
\address{Delft Institute of Applied Mathematics, Delft University of Technology, Mekelweg 4, Delft 2628DC, Netherlands}
\email{j.l.a.dubbeldam@tudelft.nl}

\begin{abstract}
Adaptive therapy is a recent paradigm in cancer treatment aiming at indefinite, safe containment of the disease when cure is judged unattainable. In modeling this approach, inherent limitations arise due to the structure of the vector fields and the bounds imposed by toxic side-effects of the drug. In this work we analyze these limitations in a minimal class of models describing a cancer population with a slowly co-evolving drug resistance trait. Chemotherapeutic treatment is introduced as any bounded time-varying input, forcing the cells to adapt to a changing environment. We leverage the affine structure and low dimension of the system to explicitly construct controllable subsets of the state space enclosing sets of equilibria. We show that these controllable sets entirely determine the asymptotic behavior of all trajectories that cannot lead to a cure.

\end{abstract}



\maketitle

\section{Introduction}\label{sec:intro}

Drug resistance remains a major obstacle in the design of treatment strategies for metastatic cancers \cite{gatenby2020}.
Chemotherapeutic agents particularly effective in the eradication of cancer cells often come with similarly devastating side-effects to healthy tissues.
Aggressive treatment strategies at the maximum acceptable level of drug toxicity are often doomed to fail as the cancer cells evolve resistance quickly enough to avoid extinction -- so-called evolutionary rescue \cite{alexander2014}.
In response, mathematical modeling of cancer therapy has surged over the last decade \cite{brady2019}, with many alternative treatment strategies being proposed.
The lack of clinical data necessary to resolve the dynamics in higher dimensions has popularized the use of planar ordinary differential equations describing population densities or frequencies, using concepts from optimal control and evolutionary game theory (for the latter see e.g. \cite{wolfl2022, gluzman2020}).
To incorporate resistance, some of these approaches assume the cancer cell population to be separable into fully drug resistant and sensitive compartments \cite{viossat2021,carrere2017}, including transitions between compartments based on drug dosage levels \cite{alvarez2024,ledzewicz2006}.
In others, resistance evolves instead as a continuous trait in response to the treatment \cite{kareva2022,pressley2021}.
This latter approach has its formal basis in the so-called canonical equation of adaptive dynamics \cite{dieckmann1996,champagnat2002}, derived as a deterministic limit from a Markov jump process. Similar dynamics have been derived in the standard formulation of quantitative genetics \cite{lande1976}, fitness generating functions \cite{vincent2005book} and were conceptualized already in \cite{waddington1957}.
In all cases, the resulting trait-dynamics are gradient flows, likening the process of evolution to a hill climbing process in an adaptive landscape.

\begin{figure}[ht]
 \begin{center}
  \includegraphics[width=.8\linewidth]{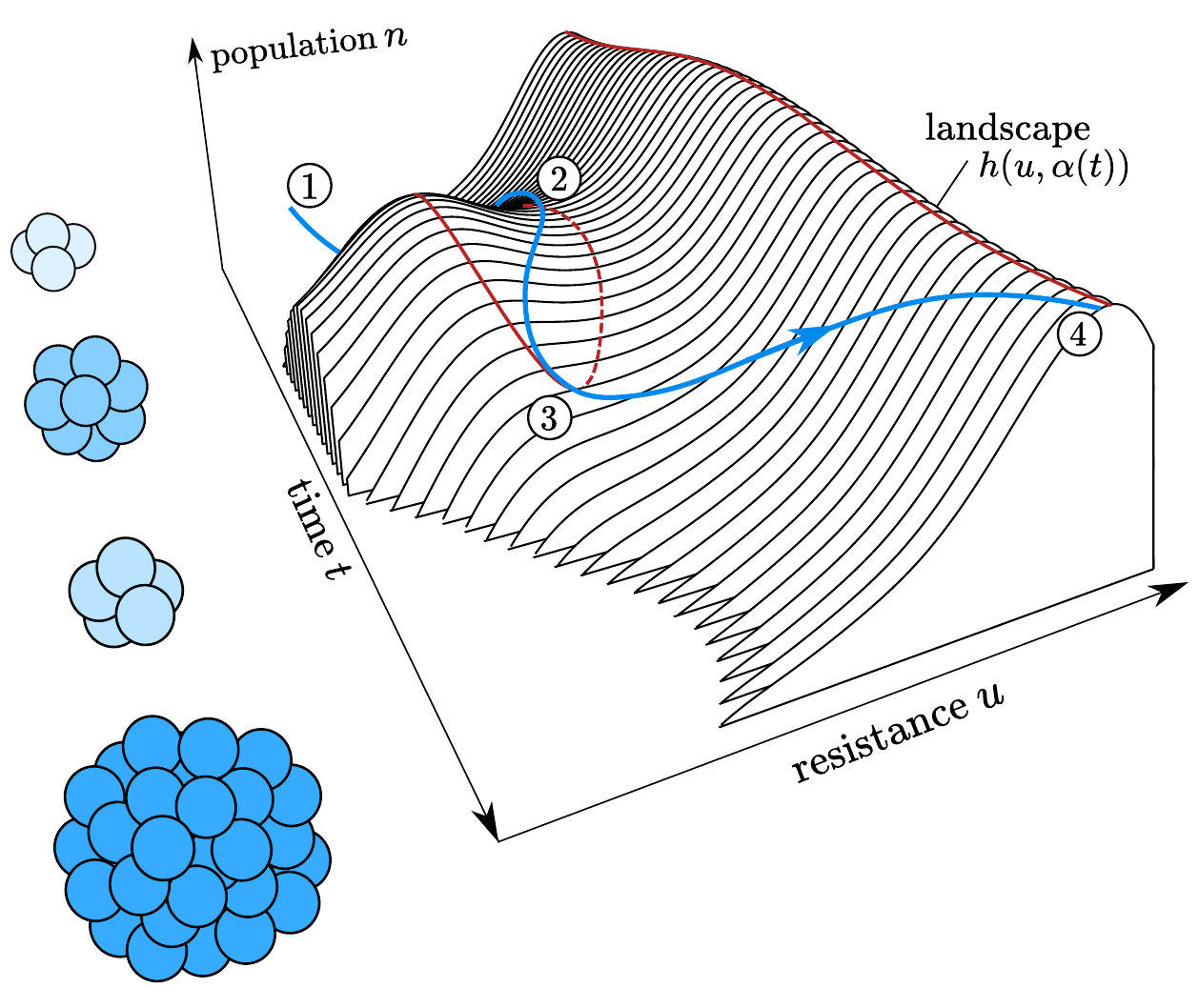}
  \caption{Sketch of the approach taken in this work. The landscape defined by a function $h(u,a)$ with ridges (solid red lines) and valleys (dotted red line) appearing and vanishing as the drug dosage of the chemotherapy is varied over time (cp. Waddington's classical picture \cite{ferrell2012}). A treatment protocol thus yields a sequence of landscapes with respect to which the cancer population tries to optimize its resistance. To do so, the population chases a moving (local) maximum of the landscape. In this example the treatment is decreased and then increased again. The resulting patient trajectory is sketched in blue. The cell population corresponding to the trajectory is sketched at four stages with the intensity of the blue color illustrating increasing and decreasing resistance. }
  \end{center}
  \label{fig:3d}
 \end{figure}

Although low-dimensional systems can reflect real-world applications with only limited accuracy, their relative simplicity allows for detailed analysis and the insights gained in recent years have helped pave the way for clinical trials \cite{zhang2017}. 
Preliminary results from these trials have led researchers to the conjecture that even dangerous metastatic cancers can be overcome by abandoning the goal of cure in favor of life-long containment of the disease \cite{west2023}.
In this case, the patient trajectory is to be confined within a ``safe'' subset of the state space while simultaneously optimizing a performance index (``quality of life'').
A simple example for a safe set and quality of life is placing an upper bound on the tumor volume and the cumulative drug dosage over time (see e.g. \cite{Carrere2019}).
A question unaddressed so far, is how such sets and measures can be defined in general. 
That is, how are candidate subsets for long-time, optimal containment limited by the structure of the vector fields used to model the disease and bounds imposed on the efficacy of the treatment?
Generally, this question can be approached using geometric control \cite{jurdjevic1997} and viability theory \cite{aubin1984}.
Explicitly finding the relevant subsets, however, is a very challenging task usually achieved only numerically, if at all (see e.g. \cite{szolnoki2003}).
For the special case of planar systems, the situation is more tractable, as both autonomous \cite{jones1969} and set-valued \cite{hautus1977} dynamical systems alike generically obey Poincar\'{e}-Bendixson type results.
Subject of this work is to demonstrate the limitations of containment in a minimal class of models with a simple bifurcation diagram.
The model consists of a cancer population and slowly co-evolving drug resistance trait under the effect of any bounded time-varying input.
It is a generalization of \cite{pressley2021,kareva2022} and incorporates a variety of ways in which resistance may alter the fitness of the cancer cells.
We explicitly construct controllable subsets from simple unions of orbits for autonomous vector fields, surrounding connected components of the set of equilibria. 
Within these sets, the trajectory can be adjusted as desired by the input.
The connection between controllability and sets of equilibria was studied previously by Nikitin \cite{nikitin1994} and Sun \cite{sun2007}.
We show that these sets entirely determine the long-time behavior of the system and any subset disjoint from controllable sets is escaped from in finite time.
Our methods are drawn from the theory of planar differential inclusions as summarized by Filippov \cite{Filippov1988}, and the global theory of control systems developed extensively by Colonius and Kliemann \cite{ColoniusKliemann2000}, simplified to the context of our model.

The paper is structured as follows: In Section~\ref{sec:model} we introduce our class of model and assumptions. In Section~\ref{sec:prelim} we determine the consequences of these assumptions on the dynamics. Section~\ref{sec:palliative} contains our main results, where we determine the influence of any time-varying treatment input on the dynamics for initial states from which it is impossible to cure. 
Finally, we summarize and conclude our results in Section~\ref{sec:conclusion}.

\section{\label{sec:model}Model}

We consider a logistic equation for the total number of cells $n$ comprising an underlying, more heterogeneous, population of cancerous cells with varying levels of resistance to a specific chemotherapy drug.
The expected level of drug resistance across the collection of cells is characterized by a dimensionless trait-value $u$.
The trait-value represents any of the complex molecular mechanisms negating the effect of the drug on the population \cite{holohan2013}.
The drug is modelled by the time-varying input $\alpha : [0,\infty) \rightarrow A$, taking values on the range $A \subset [0,\infty)$.
To account for its toxic side effects, when considering feasible inputs we restrict to a (large) bounded interval $A \subset [0,a_M]$, where $0$ corresponds to no treatment and $a_M$ represents the instantaneous maximum tolerable dose.
Exposure to the drug exerts a selective force on the cell population to adapt their resistance through mutation in order to maximize the per-capita growth rate in the given environment. 
According to classical evolutionary theory, the dynamics of the total cell population thus generate co-evolutionary dynamics of the expected trait value. At leading order, the selective force is given by the gradient of the growth rate \cite{dieckmann1996}. 
Our general model is the following planar system of nonautonomous ordinary differential equations:
\begin{align} \label{sys:full}
    \begin{cases}
     \dot{u} = \varepsilon k(n) \pdv{u} (b(u,\alpha(t)) - c(u) n) \\
     \dot{n} = n \big( b(u,\alpha(t)) - c(u) n \big),
    \end{cases}
\end{align}
under any discontinuous time-varying input $\alpha$ from the class
 \begin{equation*}
     \mathcal{A} = \{ \alpha: [0,\infty) \rightarrow A \mid \alpha \text{ piecewise constant}\}.
 \end{equation*}
In (\ref{sys:full}), the rate parameter $\varepsilon = \frac{\mu \sigma^2}{2}$ is given by the fraction of births resulting in a successful mutation $\mu \geq 0$, and the variance of mutation from the expected trait-value $\sigma^2$.
The trait dynamics are coupled to the population via the evolutionary rate function $k(n)$. 
It has previously been taken as either the identity \cite{dieckmann1996} or, in the cancer context, a constant \cite{pressley2021, Traulsen2023}. 
We instead assume here that it is $C^1$ over $\mathbb{R}$ and vanishes only when the population is extinct:
\begin{enumerate}[label=\textbf{H\arabic*.},ref=H\arabic*]
    \item \label{H-K} $k^{-1}(0)=\{0\}$ and $k'(0)>0$, 
\end{enumerate} 
where $'=\frac{\partial}{\partial u}$.
The extension of the rate function to negative population $n<0$ is arbitrary and purely for formal convenience later (see system~(\ref{sys:reduced})).
We consider a class of growth functions $b: \mathbb{R} \times A \rightarrow \mathbb{R}$ that are $C^\infty$ in $(u,a)$ under additional assumptions formulated in the following:
The most common hypothesis for chemotherapy is that the administered drug eliminates a certain proportion of cells, independently of the total population size (\textit{log-kill hypothesis}, \cite{norton1986}). This yields the following affine form with respect to the treatment effect:

\begin{enumerate}[resume,label=\textbf{H\arabic*.},ref=H\arabic*]
    \item \label{H-logkill}  $b(u,a) = b_0(u) - a b_1(u)$, where $b_1(u) > 0$ and $b'_1(u)<0$. 
\end{enumerate}
The function $b_1$ defines how the efficacy of the treatment is mitigated by the trait.
For the log-kill hypothesis a treatment $a>0$ always negatively impacts the cell growth.
Drug efficacy strictly decreases with increasing resistance. While populations at any value of resistance are effected to some degree, lower trait-values are targeted more strongly.
If taken as a step-function, vanishing beyond some threshold value in $u_c$, $b_1$ would instead define two compartments, one of sensitive ($u < u_c$) and one of completely resistant ($u \ge u_c$) cells with no co-evolution.
The function $b_0$ describes the trade-off between resistance and intrinsic growth rate of the cells.
The details of how this trade-off occurs are poorly understood and may be strongly patient specific.
We consider here the general setting of arbitrary $b_0$ and discuss in our examples several different possible mechanisms (see Section~\ref{sec:examples1}).
We then define a growth- or fitness landscape determined by the function
\begin{equation}\label{eq:h}
    h(u,a)= \frac{b(u,a)}{c(u)},
\end{equation}
where the higher order interaction function $c: \mathbb{R} \rightarrow (0,\infty)$ is $C^\infty$ and positive to avoid unbounded growth.
The affine form \ref{H-logkill} dictates the way in which the landscape changes under the effect of treatment.
The cancer cells continuously adapt to these changes, chasing the maxima of $h(u,\alpha(t))$ (see Figure~\ref{fig:3d}).
We assume that for increasing treatment values these maxima shift always toward higher values of resistance.
This is enforced by the relation
\begin{enumerate}[resume,label=\textbf{H\arabic*.},ref=H\arabic*]
    \item \label{H-nonlinear} $\displaystyle{ \frac{c'(u)}{c(u)} > \frac{b'_1(u)}{b_1(u)}}$ for all $u\in \mathbb{R}$.
\end{enumerate}
Lastly, we assume that fitness maxima are attained at only finitely many trait-values $u$ and that the landscape eventually decays to negative infinity. 
This means that the evolutionary dynamics select only for values on an interval, outside of which the landscape forms a barrier.
If, for example, the trait-value represents the physical size of the cells (cp. \cite{carrere2017}), maxima of the fitness must be found only at positive values with a prohibitively high death-rate below $u=0$.
\begin{enumerate}[resume,label=\textbf{H\arabic*.},ref=H\arabic*]
    \item \label{H-dissipative} $h(\cdot,a)$ has a finite number of critical points, isolated for all $a \in \mathbb{R}$ and is radially unbounded, $h(u,a) \rightarrow -\infty$ for $|u|\rightarrow \infty$ for all $a \in [0,\infty)$.
\end{enumerate}

\begin{remark}
    Convergence of the Markov process to the deterministic gradient flow describing trait dynamics occurs in the limit of small jumps $\sigma^2 \rightarrow 0$ \cite{champagnat2002}.
    In system~(\ref{sys:full}), $\varepsilon$ therefore acts as a small timescale separation parameter $0 < \varepsilon \ll 1$.
    On the slow timescale $\tau = \varepsilon t$ we recover the canonical equation of adaptive dynamics in the limit $\varepsilon \rightarrow 0$. Methods from geometric singular perturbation theory will be useful in the study of (\ref{sys:full}) for fixed $a \in A$ later. 
\end{remark}

\begin{figure}[ht]
 \begin{center}
  \includegraphics[width=1\linewidth]{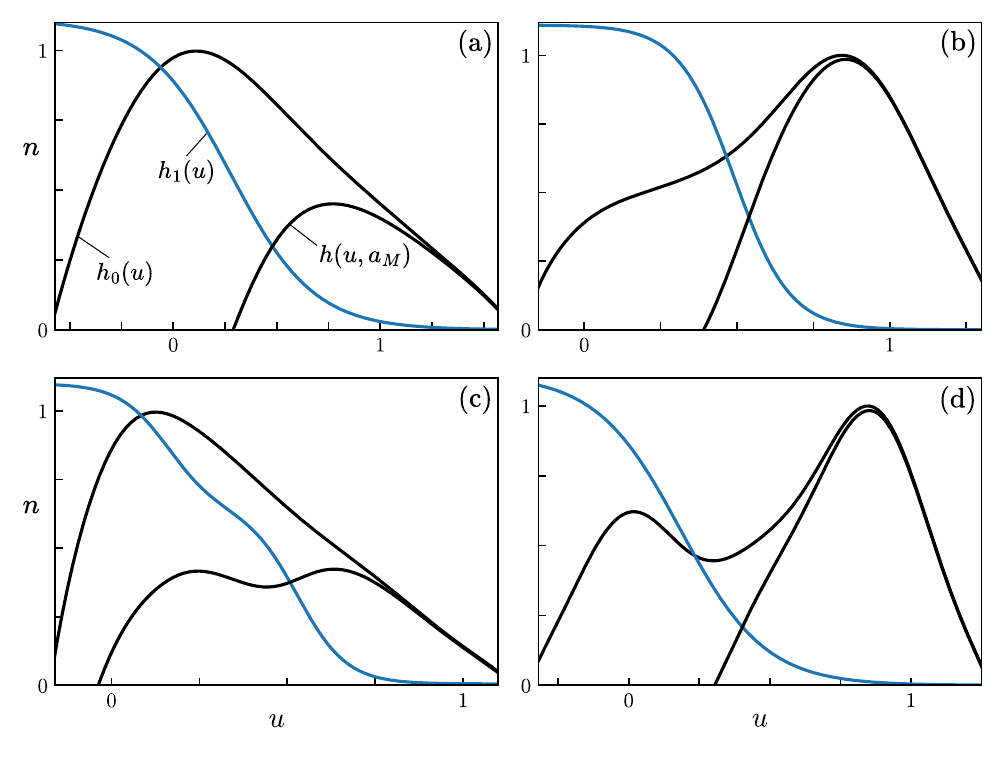}
  \caption{Examples of growth landscape $h(u,a) = h_0(u) - a h_1(u)$ \eqref{eq:h} for treatment values $a \in  \{0,a_M\}$. (a) Simple cost of resistance landscape. (b) An example without cost of resistance, where the growth landscape takes a unique maximum for high trait-values, remaining nearly unaffected by the treatment. (c) A more complex cost of resistance example. Under the effect of treatment, a single global maximum of the landscape is deformed into a ``double well''-type potential at $a = a_M$. (d) A landscape with partial cost of resistance. It has a global maximum at high resistance for $a=0$, shielded by a local maximum near $u=0$. Parameter values and functions used can be found in Appendix~\ref{app:parameters}. }
  \end{center}
  \label{fig:potentials}
 \end{figure}

\section{Preliminary results}\label{sec:prelim}

We write the system~(\ref{sys:full}) more compactly as an affine control or switching system
\begin{equation}\label{sys:general}
    \dot{x} = f(x,\alpha(t)) = f_0(x) + \alpha(t) f_1(x).
\end{equation}
Components are denoted as $x = (u,n)$.
According to \ref{H-logkill} the vector field is split into parts given by
\begin{align}\label{eq:f0f1}
    f_0(x) = \mqty(\varepsilon k(n) (b'_0(u) - c'(u)n) \\ n (b_0(u) - c(u) n)),  \quad f_1(x) = -\mqty(\varepsilon k(n) b'_1(u) \\ n b_1(u)).
\end{align}
Dynamics of the system are determined by the family of $C^1$ vector fields $F = \{ f(\cdot,a) \mid a \in A\}$, where $f(\cdot,a):X \rightarrow \mathbb{R}^2$ for $X = \mathbb{R} \times [0,\infty)$. Each system $\dot{x}=f(x,a)$ generates a flow $\phi^a : \mathbb{R} \times X \rightarrow X$ that maps an initial datum $x_0$ to the solution at time $t$ given by $\phi^a(t,x_0)=\phi^a_t(x_0)$.
The forward orbit under any flow $\phi$ from a point $x_0$ is given by the image
\begin{equation*}
    \gamma(x_0,\phi) = \{ \phi_t(x_0) \mid t \geq 0 \}.
\end{equation*}
Similarly, we denote by $\Wstab(p,\phi), \Wunst(p,\phi)$ the stable and unstable manifold of equilibria $p \in X$ for $\phi$.
Central to our analysis is the following generalization of the forward orbit to families of vector fields.
\begin{definition}[cf. \cite{ColoniusKliemann2000}]\label{def:reachable}
    The \textup{reachable set} from a point $x_0 \in X$ under the family $\{f(\cdot,a) \mid a \in A\}$ is given by 
    \begin{equation*}
     \Gamma(x_0) = \{ \phi^{a_m}_{t_m}  \circ ... \circ \phi^{a_1}_{t_1}(x_0) \mid t_j \geq 0,  a_j \in A, m \in \mathbb{N} \}.
    \end{equation*}
\end{definition}
For a given initial condition $x_0$ and an input $\alpha \in \mathcal{A}$ there exists a unique solution of (\ref{sys:general}) in the sense of Carath\'{e}odory \cite{coddington1955}. 
It is absolutely continuous and appears stitched together from a sequence of Cauchy problems (see Fig.~\ref{fig:reachable}).
The reachable set is the union of all possible such sequences.
The treatment sequence $(a_j)$ corresponds to a sequence of adaptive landscapes $(h(\cdot,a_j))$ with respect to which the trait will evolve (see Fig.~\ref{fig:3d}).
We define the cemetery state 
\begin{equation}\label{eq:E}
 E = \{(u,0) \mid u \in \mathbb{R} \}.
\end{equation}
By \ref{H-K} this is a line of equilibrium solutions for any $a\in \mathbb{R}$. 
It is useful to define the following simplified system, in which $E$ no longer consists of solutions:
 \begin{align}\label{sys:reduced}
         \begin{cases}
            \dot{u} = \varepsilon \left(b'(u,\alpha(t)) -c'(u)n\right) \\
            \dot{n} = \frac{1}{\widetilde{k}(n)} \left( b(u,\alpha(t)) - c(u)n \right),
        \end{cases}
     \end{align}
where $\widetilde{k}(n) = k(n)/n$. 
In the following we will refer always to (\ref{sys:full}) as the full system and to (\ref{sys:reduced}) as the reduced system. 
The reduced family $\widetilde{F} = \{ \widetilde{f}(\cdot,a) \mid a \in A \}$ is considered on an extended state space $\ftil(\cdot,a):\Xtil \rightarrow \mathbb{R}^2$, where $\Xtil = \mathbb{R}^2$ by assumption \ref{H-K}. 
We denote by $\widetilde{\phi}^a$ the reduced flow and reachable sets using $\widetilde{\Gamma}$.
Up to intersections with $E$, reachable sets for systems (\ref{sys:full}) and (\ref{sys:reduced}) are the same (see Fig. \ref{fig:reachable}). 
\begin{lemma}\label{prop:equivalence}
    For the killing time $\tau_E:X \times A \rightarrow [0,\infty]$,
    \begin{equation*}
        \tau_E(x_0,a) = \inf \{ t \geq 0 \mid \phitil^a_t(x_0) \in E \},
    \end{equation*}
    it holds that
    \begin{gather*}
        \gamma(x_0,\phi^a) = \{ \phitil^a_t(x_0) \mid 0 \leq t < \tau_E(x_0,a) \}, \\
        \Gamma(x_0) = \{ \phitil^{a_m}_{t_m}  \circ ... \circ \phitil^{a_1}_{t_1}(x_0) \mid 0 \leq t_{j} < \tau_E(x_{j-1},a_{j}),  a_j \in A, m \in \mathbb{N} \},
    \end{gather*}
    where $x_j = \phitil^{a_j}_{t_j} (x_{j-1})$.
\end{lemma}
\begin{proof}
    A trajectory of the reduced system under the input $\alpha \in \mathcal{A}$ corresponds to the sequence $(x_j)_{j \in J} = (\phitil^{a_j}_{t_j} (x_{j-1}))_{j \in J}$, where $J \subset \mathbb{N}$.
    Let $\phitil^a_t(x)$ be an element of this sequence.
    From \ref{H-K} we can factor $k(n) = n \widetilde{k}(n)$, where $\widetilde{k}>0$. Then $f(x,a)=k(n) \ftil(x,a)$ such that, for fixed $a$, full and reduced vector field are the same up to a $C^1$ function that preserves the direction of time for $n>0$.
    Times $t$ and $s$ of the reduced and full system respectively are related as 
    \begin{equation}\label{eq:stime}
        s(t,x,a) = \int^{t}_0 \frac{d\tau}{k([\phitil^a_\tau(x)]_n)},
    \end{equation}
    where $[\cdot]_n$ denotes the $n$-component.
    The integral exists for $t < \tau_E(x,a)$ and we write 
    \begin{equation*}
        x_j = \phitil^{a_j}_{t_j}(x_{j-1}) = \phi^{a_j}_{s_j}(x_{j-1}),
    \end{equation*}
    where $s_j = s(t_j,x_{j-1},a_j)$, for any element of the sequence with $t_j < \tau_E(x_{j-1},a_j)$.
\end{proof}

\begin{figure}[ht]
 \begin{center}
  \includegraphics[width=0.8\linewidth]{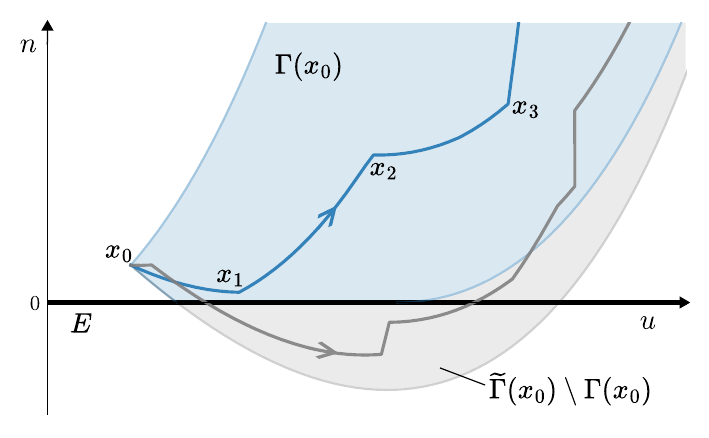}
  \caption{Comparison of reachable sets for the full system $\Gamma(x_0)$ and reduced system $\widetilde{\Gamma}(x_0)$ from a point $x_0 \in S$; see (\ref{eq:S}). The blue curve is a (absolutely continuous) sample trajectory traversing $x_j = \phi^{a_j}_{t_j}(x_{j-1})$. Because this trajectory does not intersect $E$, it is the same for full and reduced system under the rescaling of time (\ref{eq:stime}). The grey sample trajectory intersects $E$ for $t=\infty$ in the full system and (twice) in finite time for the reduced system.}
  \label{fig:reachable}
  \end{center}
 \end{figure}

By Lemma~\ref{prop:equivalence}, initial states for which the reachable set for the full (\ref{sys:full}) and reduced system (\ref{sys:reduced}) do not coincide is given by
\begin{equation}
    S = \{ x_0 \in X \mid \widetilde{\Gamma}(x_0) \cap E \neq \emptyset \}. \label{eq:S}
\end{equation}
From points in $S$ the cell population fails to adapt enough to avoid extinction under certain inputs $\alpha \in \mathcal{A}$, limiting to the cemetery state $E$ under the dynamics of the full system~(\ref{sys:full}).
Due to the $u$-component evolving as the gradient of the $n$-component, the only nonempty limit sets for the family $\phitil^a$ are equilibria.
By (\ref{sys:reduced}), these are critical points of the function $h(\cdot,a)$ (\ref{eq:h}) that defines the growth landscape for the cancer cells.
We define the manifold of equilibria
\begin{align}
    \mathcal{P} &= \{ (x,a) \mid \ftil(x,a) = 0 \} \subset \mathbb{R}^2 \times \mathbb{R} \nonumber \\
                &= \{ (x,a) \mid a = a_\star(u) \text{ and } n = h_\star(u) \}, \label{eq:Pgraph}
\end{align}
where $h_\star(u) = h(u,a_\star(u))$.
Equation~(\ref{eq:Pgraph}) follows from \ref{H-logkill} and \ref{H-nonlinear}, using the implicit function theorem to write $\mathcal{P}$ globally as the graph of a smooth function parametrized by $u$.
Restricting to the bounded input range $A \subset [0,a_M]$ yields the subset of feasible equilibria $\mathcal{P}_A = \mathcal{P}\rvert_{a \in A}$.
By \ref{H-dissipative}, $h(\cdot,a)$ has a finite number of critical points for all $a \in A$ and is eventually monotone.
The set $\mathcal{P}_A$ thus lies contained in a compact subset $U \times N \times A \subset \mathbb{R}^3$ (see Figure~\ref{fig:PA}(c)).
The simple bifurcation diagram for the family $\widetilde{F}$ is characterized in the following.

\begin{lemma}\label{prop:partitionP}
    The set of equilibria $\mathcal{P}$ can be partitioned as
    \begin{align*}
        \Pstab &= \{ (x,a) \in \mathcal{P} \mid a_\star'(u) > 0 \} \neq \emptyset, \\
        \Punst &= \{ (x,a) \in \mathcal{P} \mid a'_\star(u) < 0 \}, \\
        \Pbif &= \{ (x,a) \in \mathcal{P} \mid a'_\star(u) = 0 \},
    \end{align*}
    into stable nodes, saddles and isolated nonhyperbolic points of the corresponding flows $\phitil^{a_\star}$. 
    The subset $\{(x,a) \in \Pbif \mid a''_\star(u) \neq 0 \}$ consists of generic fold bifurcation points of the corresponding flows $\phitil^{a_\star}$.
\end{lemma}

 \begin{figure}[ht]
 \begin{center}
  \includegraphics[width=1\textwidth]{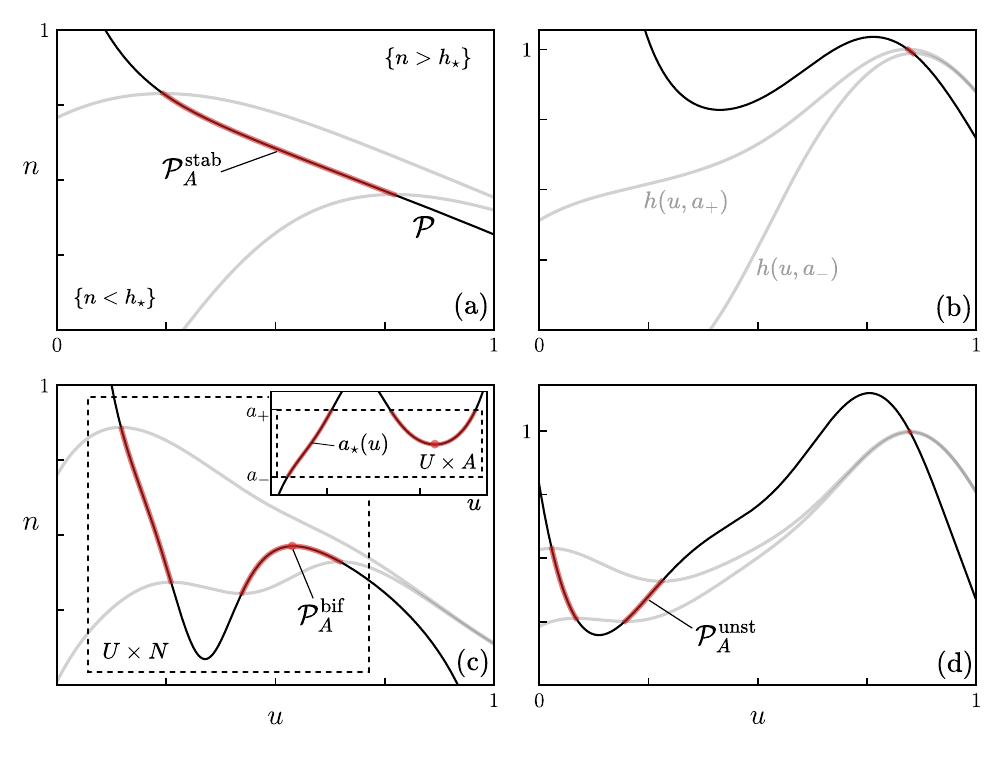}
  \caption{Examples for the manifold of equilibria $\mathcal{P}$ and feasible equilbria $\mathcal{P}_A$ projected on $X$ for a chosen bounded treatment range $A = [a_-,a_+] \subset [0,a_M]$. Function $h(u,a)$ chosen as in Fig.~\ref{fig:potentials}.
  (a) Connected $\mathcal{P}_A = \Pstab_A$ consisting only of stable nodes. (b) Stable node largely unaffected by the treatment range.
  (c) Disconnected $\mathcal{P}_A$ in two components, one comprised only of stable nodes, the other of nodes, saddles and an isolated fold bifurcation point (red dot). Dashed box shows a possible compact set $U \times N$ containing all feasible equilibria. Inset shows equilibria as $(u,a_\star(u))$ and $U \times A$. In the context of section~\ref{sec:palliative}, the two components $\IA$ are of type (1) and (3). 
  (d) Disconnected $\mathcal{P}_A$ in three components, comprised of only nodes, or only saddles. Right most component is nearly invisible due to small effect of treatment.}
  \label{fig:PA}
  \end{center}
 \end{figure}

 \begin{proof}
     Differentiating $\ftil(u,h_\star(u),a_\star(u))=0$ in $u$ yields (in shorthand)
     \begin{equation*}
         a'_\star = \frac{c(b''-c''h_\star)}{b'_1c-b_1c'}.
     \end{equation*}
     Linearization along $\mathcal{P}$ is thus given by
     \begin{equation}
         D_x \widetilde{f}\rvert_\mathcal{P} = \mqty(  \varepsilon (b''- c''h_\star) & -\varepsilon c' \\
                        0 &  - \frac{c}{\widetilde{k}(h_\star)} ) 
                        = \mqty(  \varepsilon \frac{b'_1c-b_1c'}{c} a'_\star  & -\varepsilon c' \\
                        0 &  - \frac{c}{\widetilde{k}(h_\star)} ), \label{eq:Jac}
     \end{equation}
     with spectrum we can write as $\{-\lambda_1(u),\varepsilon \lambda_2(u) a'_\star(u)\}$ where $\lambda_2(u)<0$ by \ref{H-nonlinear} and $\lambda_1(u)=c(u)/\widetilde{k}(h_\star(u))>0$ by \ref{H-K}.
     The stability is thus determined by $a'_\star(u)$, both eigenvalues being negative for $a'_\star(u) > 0$ (stable nodes) and having opposite signs for $a'_\star(u) < 0$ (saddles).
    By dissipativity \ref{H-dissipative}, the set of nodes is non-empty.
    An equilibrium $p=(u_1,n_1)$ is nonhyperbolic only when $a'_\star(u_1)=0$.
    By \ref{H-dissipative} such points are isolated.
    The tangency of the graph $(u,a_\star(u))$ to $\{a = a_\star(u_1)\}$ is quadratic for $a_\star''(u_1) \neq 0$ and we have a generic fold bifurcation.
    If instead this term vanishes, there is either no qualitative change to the flow (leading order derivative $a^{(2m+1)}_\star(u_1) \neq 0$, $m\ge1$) or a degenerate fold bifurcation (leading order derivative $a^{(2m)}_\star(u_1) \neq 0$, $m>1$).   
 \end{proof}

 Of particular interest in Section~\ref{sec:palliative} are ranges $A$ for which the set of nonhyperbolic points $\Pbif_A$ does not contain equilibria corresponding to the extreme values of the input.

\begin{definition}\label{def:hyperbolic}
    The input range $A = [a_-,a_+]$ is called \textup{hyperbolic}, if all equilibria for the vector fields $\ftil(\cdot,a_\pm)$ are hyperbolic.
\end{definition}

For the full vector field $f(\cdot,a)$ there exist additional nonhyperbolic points, where an interior equilibrium collides with the cemetery state $E$.
From (\ref{eq:Pgraph}), the graph $H^\star=(u,h_\star(u))$ partitions the state space into ``upper" and ``lower" components denoted $\{n > h_\star\}$ and $\{ n < h_\star \}$ (see Figure~\ref{fig:PA}(a)).
Due to the affine structure (\ref{sys:general}) the relative orientation of vector fields with different values of $a$ must remain fixed with respect to each other throughout the upper and lower component, see Lemma~\ref{prop:angle_condition}. 
In the following let $\inner{\cdot}{\cdot}$ denote the standard Euclidean inner product.
For any $v = (v_1,v_2)^T$ we define an orthogonal vector by the convention $v^\perp = (-v_2,v_1)^T$.
\vspace{0.2cm}
\begin{lemma} \label{prop:angle_condition}
    Let $a_1,a_2 \in \mathbb{R}$ and $a_1 < a_2$. Then it holds that
    \begin{equation*}
       \sgn \inner{f(x,a_1)^{\perp}}{f(x,a_2)} = - \sgn \inner{f(x,a_2)^{\perp}}{f(x,a_1)}  = \begin{cases}
                                                    + 1 \textup{ for } x \in \{ n > h_\star \} \\
                                                    - 1 \textup{ for } x \in \{ n < h_\star \}.
                                                \end{cases}
    \end{equation*}
\end{lemma}

\begin{proof}
    Follows directly from
    \begin{align*}
        \inner{f(x,a_1)^{\perp}}{f(x,a_2)} &= a_1 \inner{f_1(x)^{\perp}}{f_0(x)} + a_2 \inner{f_0(x)^{\perp}}{f_1(x)} \\
        &= (a_2 - a_1) \inner{f_0(x)^{\perp}}{f_1(x)},
    \end{align*}
    and $\sgn \inner{f_0(x)^{\perp}}{f_1(x)} = \sgn (n - h_\star(u))$ by \ref{H-logkill} and \ref{H-nonlinear}. The set $H^\star =\{ n = h_\star \}$ by construction consists of those points where $f_0$ and $f_1$ are linearly dependent (see (\ref{eq:f0f1})).
\end{proof}

This simple property allows making statements about the entire family of vector fields $F$ (and thus $\Gamma$) by considering only the autonomous vector fields $f(\cdot,a_\pm)$ for extreme values of the input.

\subsection{Examples}\label{sec:examples1}

We consider four examples\footnote{For the functions and parameter values used in the examples, see Appendix~\ref{app:parameters}.} (a)--(d) of growth landscapes and the resulting manifolds of equilibria on $X$, shown in Figures~\ref{fig:potentials} and {\ref{fig:PA}.
They describe four different ways in which treatment, resistance and population growth may interact.
By construction, the components of the vector field $f_1$ (\ref{eq:f0f1}) have fixed opposite sign throughout $X \setminus E$, treatment always reducing population size to some extent, but inducing resistance.
On the other hand, depending on the general function $b_0$, there can exist subsets of $X$ where the treatment independent vector field $f_0$ drives up the cell population while re-sensitizing them to the drug.
Trade-offs of this type are used in adaptive therapy \cite{gatenby2020} to prolong responsiveness to therapy when cure is assumed impossible.
It has been proposed that any increase in resistance should result in a concomitant decrease of the growth rate (\textit{cost of resistance}, \cite{gallaher2018}).
A simple example of this mechanism is (a).
This is the landscape studied in \cite{pressley2021}, where a unique global maximum of $h(\cdot,a)$ exists for all treatment values $a \in [0,a_M]$ (a stable node for the corresponding autonomous flow).
The stable node shifts from low resistance values (as defined by $b_1$) near $u=0$ for $a=0$ toward high resistance for $a=a_M$ along $\Pstab_A$.
Experimental results on cost of resistance remain inconclusive with some evidence suggesting the opposite is true \cite{kaznatcheev2019}, i.e. resistance may either increase without the drug or may continue increasing after the drug is removed.
Example (b) shows a system of the former type, where the global maximum is always attained at high values of resistance and there is no mechanism to re-sensitize the cells. 
Applying treatment serves only to further increase the selective force toward high resistance and the stable node moves very little under the effect of the drug, $b_1(u)$ having decayed close to zero.
A related case was studied in \cite{greene2018}, where resistance evolution is assumed irreversible and any treatment eventually results in a fully resistant population. 
Example (c) shows a more complex cost of resistance scenario. For increasing treatment, the landscape with a single global maximum at low resistance is deformed into a bistable one. A new locally optimal trait-value emerges at high resistance. By Lemma~\ref{prop:partitionP} there is a generic fold bifurcation of the autonomous flow $\phi^{a_\star}$ (red dot, Figure~\ref{fig:PA} (c)).
Our final example (d) has a local cost of resistance mechanism that is destroyed for high values of treatment.
In contrast to example (b), there is a low resistance maximum of $h(\cdot,0)$ separated from the globally optimal trait at high resistance by a minimum (saddle of the autonomous flow). 
When the treatment is not too high, in a neighborhood of the left local maximum the system behaves like example (a).
For sufficiently high treatment value, the low-resistance maximum vanishes and only a single optimum remains. Resistance now increases after therapy is stopped and the system then behaves like example (b).

\section{Main results}\label{sec:palliative}

We determine the overall influence the input has on the solutions of the system~(\ref{sys:full}) on $X\setminus S$, i.e. for states from which cure is unattainable under any constant-in-time or time-varying input.
For a bounded input range $A = [a_-,a_+]$, the set of feasible equilibria $\mathcal{P}_A$ is partitioned into its connected components $I_{\mathcal{P}_A}$, i.e. maximal connected subsets.
When $A$ is hyperbolic, see Definition~\ref{def:hyperbolic}, then by Lemma~\ref{prop:partitionP} endpoints of the connected components must be equilibria for extreme values $a \in \{a_\pm\}$ (by maximality) and must appear in three types: (1.) both endpoints are stable nodes, for different values of $a$, (2.) both endpoints are saddles, for different values of $a$, and (3.) one point is a stable node, the other a saddle point for the same value of $a$.
Examples for all three types can be seen in Figure~\ref{fig:PA} (c)--(d).
We make use of the following definition of an enclosed set.

\begin{definition}[cf. \cite{hautus1977}]
    Given a simple closed curve $\Omega$ the set $\enc \Omega$ is the open set given by the bounded component of $\mathbb{R}^2 \setminus \Omega$.
\end{definition}

For sufficiently small rate parameter $\varepsilon$, each of the three types of connected component can be enclosed by a curve in phase space.
The first result of this section is Theorem~1, split into three parts, where we construct these set and show they are controllable in the following sense.

\begin{definition}\label{def:control}
    A set $K \subset \mathbb{R}^d$ is called \textup{controllable} for a family $F$, if it holds that $K \subset \Gamma(x)$ for all $x \in K$.
\end{definition}

In a controllable set, any point can be driven to any other point under some input, allowing adjustments to the system state as desired.
We begin with case (1.).

\begin{manualtheorem}{1.1}\label{prop:Omega1}
    Let $A=[a_-,a_+]$ be hyperbolic and let $I_{\mathcal{P}_A}$ be a connected component with endpoints $(p_\pm,a_\pm) \in \Pstab_A$. If $\varepsilon > 0$ is sufficiently small, then for the simple closed curve
    \begin{equation*}
        \Omega_1 = \gamma(p_-,\phitil^{a_+}) \cup \gamma(p_+,\phitil^{a_-}),
    \end{equation*}
    it holds that $\widetilde{\Gamma}(x) = \enc \Omega_1$ for all $x \in \enc \Omega_1$. 
    If $\Omega_1 \cap E = \emptyset$, then the same holds for $\Gamma(x)$.
\end{manualtheorem}

\begin{figure}[ht]
     \begin{center}
      \includegraphics[width=1\textwidth]{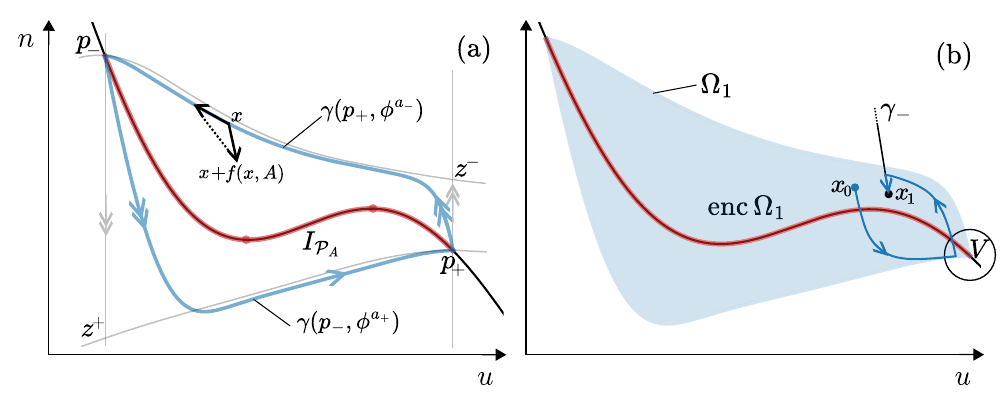}
      \caption{Example of a node-type set $\enc \Omega_1$ surrounding a connected component with endpoints $(p_\pm,a_\pm) \in \Pstab_A$. (a) Construction of the boundary $\Omega_1$ as a union of orbits. Due to Lemma~\ref{prop:angle_condition}, the image $f(x,A)$ lies in the ``lower'' tangent halfspace to $\enc \Omega_1$ at the point $x$.
      (b) The set enclosed by $\Omega_1$ with its controllability property. }
      \label{fig:omega1}
      \end{center}
 \end{figure}

In a addition to controllability, the set enclosed by $\Omega_1$ also has a forward invariance property $\widetilde{\Gamma}(\enc \Omega_1) \subset \enc \Omega_1$.
Any point in the set can be driven only to other points in the same set. Once entered, it is therefore impossible to escape under any $\alpha \in \mathcal{A}$.
We leverage smallness of the rate parameter $\varepsilon$ to ensure the set boundary can be described simply as a union of orbits for autonomous vector fields.

\begin{proof}[Proof of Theorem~\ref{prop:Omega1}]
    We work on the compact subset $U \times N \subset \mathbb{R}^2$ defined in Section~\ref{sec:model}, made large enough to contain all equilibria for $A = [a_-,a_+]$ in its interior. We then proceed in three steps:
    (i) Denote $p_\pm = (u_\pm,n_\pm)^T$. 
    By assumption, $\IA$ is transversal to the planes $\{ a =a_\pm\}$ in the same direction $a'_\star(u_\pm)>0$, so $u_-<u_+$; see Figure (\ref{fig:PA}c), inset.
    The graphs $H^\pm = \{(u,h(u,a_\pm) \mid u \in U\}$ intersect the vertical lines $\{ u = u_\mp\}$ in points $z^\pm \in \mathbb{R}^2$.
    By assumption \ref{H-logkill} the function $h$ is strictly decreasing in $a$ such that $H^\pm$ are disjoint and the points $z^\pm$ lie on opposite sides of the manifold of equilibria $H^\star$.
    Thus there exists a simple closed curve $\Omega'$ contained in the union
    \begin{equation}\label{eq:Omegaprime}
        \left(\{u=u_-\} \cup H^+\right) \cup  \left(\{u=u_+\} \cup H^-\right) .
    \end{equation}
    It follows using standard results from geometric singular perturbation theory, that the simple closed curve $\Omega_1$ is the small $\varepsilon$-perturbation of $\Omega'$.
    In the limit $\varepsilon = 0$, $H^+$ is a manifold of equilibria. Linearization yields always one zero eigenvalue and $-\lambda_1(u)<-\lambda<0$ for a constant $\lambda>0$ on $U$ (see (\ref{eq:Jac})), so $H^+$ is uniformly normally hyperbolic and attracting.
    Trajectories under $\phitil^{a_+}$ in this limit are vertical lines foliating $U \times N$, with the forward orbit from $p_-$ connecting to the base point $z^+ \in H^+$ in $\{ u = u_-\}$.
    By Fenichel's Theorem \cite{fenichel1979}, for $\varepsilon>0$ sufficiently small, $H^+$ perturbs to the $O(\varepsilon)$-close graph $H^+_\varepsilon = \{ (u,h_\varepsilon(u,a_+)) \mid u \in U\}$ of a function $h_\varepsilon$, invariant over $U\times N$ and normally attracting for $\phitil^{a_+}$.
    Moreover, we can identify with $p_-$ a base $z^+_\varepsilon \in H^+_\varepsilon$ such that 
    \begin{equation}\label{eq:tracking}
        | \phitil^{a_+}_t(p_-) - \phitil^{a_+}_t(z^+_\varepsilon)| \le C e^{-\lambda t},
    \end{equation}
    for a constant $C>0$ for all $t \geq 0$.
    The flow of the base point on $H^+_\varepsilon$ is decoupled as 
    \begin{equation}
        \dot{u} = \varepsilon \left(b'(u,a_+)-c'(u)h_\varepsilon(u,a_+)\right).
    \end{equation}
    Because $A$ is hyperbolic there is no equilibrium for $\phitil^{a_+}$ on the strip $M^+ = (v,u_+) \times N$ containing $z^+_\varepsilon$, where $v = \sup \{ u < u_- \mid h'(u,a_+) = 0\}$.
    On $M^+ \cap H^+$ the $u$-component $\ftil_u(\cdot,a_+)$ is strictly increasing. For sufficiently small $\varepsilon>0$ this also holds on $M^+ \cap H^+_\varepsilon$ such that we have $\phitil^{a_+}_t(z^+_\varepsilon) \rightarrow p_+$. 
    From (\ref{eq:tracking}) it then follows that also $\phitil^{a_+}_t(p_-) \rightarrow p_+$.
    The segment $\{u = u_-\} \cup H^+$ in (\ref{eq:Omegaprime}) thus perturbs to the orbit $\gamma(p_-,\phitil^{a_+})$.
    Because $|h'_\star(u)|$ is bounded on $U$, a suitably small choice of $\varepsilon$ also ensures that this orbit does not again intersect the manifold of equilibria $H^\star$, remaining trapped as (see Figure~\ref{fig:omega1})
    \begin{equation}\label{eq:trapping1}
        \gamma(p_-,\phitil^{a_+}) \setminus p_- \subset \{ n < h_\star\}.
    \end{equation}
    Interchanging $a_+$ and $a_-$, the argument to show that the segment $\{u = u_+\} \cup H^-$ perturbs to the orbit $\gamma(p_+,\phitil^{a_-})$ is analogous.
    In this case we use $M^- = (u_-,w) \times N$ with $w = \inf\{ u > u_+ \mid h'(u,a_-) = 0\}$ and the orbit is trapped as
    \begin{equation}\label{eq:trapping2}
        \gamma(p_+,\phitil^{a_-}) \setminus p_+ \subset \{ n > h_\star\}.
    \end{equation}
    \\
    (ii) From (\ref{eq:trapping1}) and (\ref{eq:trapping2}), the property $\Gamma(\enc \Omega_1) \subset \enc \Omega_1$ then follows directly:
    Suppose $x_{j-1} \in \mathbb{R}^2$ such that there exists a time $t_{j}>0$ and a value $a_{j} \in A$ such that $ x_j =\phitil^{a_j}_{t_j}(x_{j-1}) \in \Omega_1$.
    Then, from Lemma~\ref{prop:angle_condition} we have $\inner{\ftil(x_j,a_\pm)^\perp}{\ftil(x_j,a_j)} \geq 0$.
    By conditions (\ref{eq:trapping1}) and (\ref{eq:trapping2}) the orientation of $\Omega_1$ is counter-clockwise in forward time with respect to $\phitil^{a_\pm}$.
    Therefore, the image $\ftil(x_j,A)$ consists of vectors either tangent to $\Omega_1$, or pointing into $\enc \Omega_1$ at $x_j$ (subtangential, see \cite{yorke1967}).
    This implies $x_{j-1} \in \mathbb{R}^2 \setminus \enc \Omega_1$.
    \\
    \\
    (iii) To show the enclosed set is controllable, we fix any pair $x_0,x_1 \in \enc \Omega_1$ and show $x_1 \in \widetilde{\Gamma}(x_0)$.
    Working backward in time from the target $x_1$, we define the orbit
    \begin{equation}
        \gamma_- = \{ \phitil^{a_+}_{-t}(x_1) \mid t \geq 0 \}. \label{eq:gamma-}
    \end{equation}
    Then $\gamma_-$ must have a transversal intersection with $\gamma(p_+,\phitil^{a_-})$, as orbits for fixed values of $a$ are disjoint and $\overline{\gamma_-}$ cannot be contained in $\enc \Omega_1$ by the Poincar\'{e}-Bendixson Theorem.
    There then exists $t_2>0$ such that $\phitil^{a_-}_{t_2}(p_+) \in \gamma_-$.
    By Lemma~\ref{prop:angle_condition}, the vector field $\ftil(\cdot,a_-)$ is nowhere tangent to $\gamma_- \cap \{ n > h_\star \}$ so by the Flow-Box Thoerem, there exists a neighborhood $V$ of $p_+$ and a map $\tau: V \rightarrow [0,\infty)$ such that $\phitil^{a_-}_{\tau(y)}(y) \in \gamma_-$ for all $y \in V$ (see e.g. \cite[Corollary 1.13]{dumortier2006}).
    Because $p_+$ is the attractor for $\phitil^{a_+}$ on $\overline{\enc \Omega_1}$, there exists a time $t_1>0$ such that $\phitil^{a_+}_{t_1}(x_0) \in V$ and the result follows.
    By Lemma~\ref{prop:equivalence}, if $\Omega_1 \cap E = \emptyset$ all statements hold equally for the full system.
\end{proof}

\begin{remark}
    $\enc \Omega_1$ is minimal with respect to the forward invariance property, i.e. it does not contain a proper subset forward invariant for the range $A$. In the context of random dynamical systems, such sets are known as minimal forward invariant sets; see \cite{homburg2006,homburg2010,lamb2023}.
    Here, it is a direct consequence of Lemma~\ref{prop:angle_condition}.
    $\enc \Omega_1$ is then also maximal with respect to the controllability property, i.e. it is not a proper subset of a controllable set for the range $A$ (cp. control sets \cite{ColoniusKliemann2000}). 
\end{remark}

We next consider case (2.) and find a set of saddle type enclosing a connected component ending in saddle points.
While its boundary remains a union of orbits for autonomous vector fields, for the saddle type this union is constructed using four segments.
Each segment is a subset of the one-dimensional stable or unstable invariant manifolds of the saddles.

\begin{manualtheorem}{1.2}\label{prop:Omega2}
    Let $A=[a_-,a_+]$ be hyperbolic and let $I_{\mathcal{P}_A}$ be a connected component with endpoints $(p_\pm,a_\pm) \in \Punst_A$. For $\varepsilon>0$ sufficiently small there exists a simple closed curve  
    \begin{equation*}
        \Omega_2 \subset  \Big[ \Wstab(p_\pm,\phitil^{a_\pm}) \cup \Wunst(p_\pm,\phitil^{a_\pm}) \Big],
    \end{equation*}
    such that $\enc \Omega_2$ is controllable for $\widetilde{F}$. 
    If $\Omega_2 \cap E = \emptyset$, then it is also controllable for $F$.
\end{manualtheorem}

\begin{figure}[ht]
 \begin{center}
  \includegraphics[width=1\textwidth]{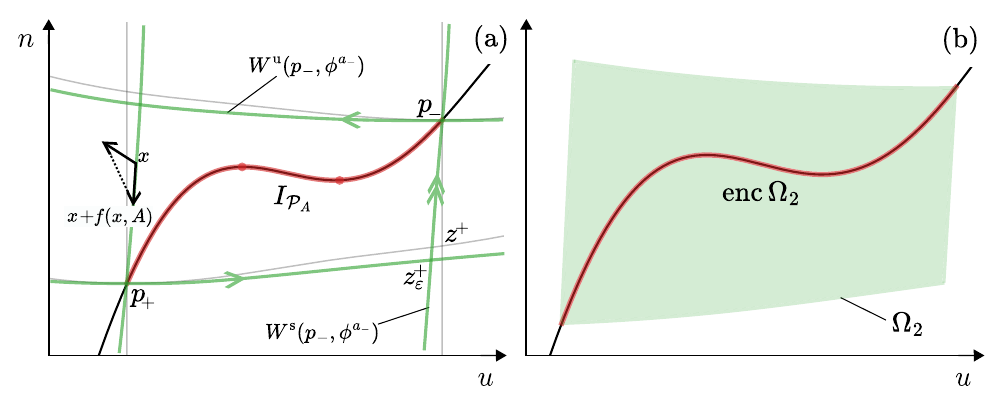}
  \caption{Example of a saddle-type set $\enc \Omega_2$ surrounding a connected component (projected onto $X$) with endpoints $(p_\pm,a_\pm) \in \Punst_A$. (a) Construction of the boundary $\Omega_2$ within the union of stable and unstable manifolds of the endpoints. The set does not have the forward invariance property as solutions can escape through the vertical boundaries given by the stable invariant manifolds where the vector fields $f(x,A)$ point either tangent to $\Omega_2$ or outside of the set $\enc \Omega_2$.
  (b) The set enclosed by $\Omega_2$ with its controllability property.  }
  \label{fig:omega2}
  \end{center}
 \end{figure}

\begin{proof}
    (i) Denote $p_\pm = (u_\pm,n_\pm)^T$. By assumption $\IA$ is transversal to the planes $\{a = a_\pm\}$ in the same direction $a'_\star(u_\pm) < 0$, so $u_+ < u_-$. 
    As in the proof of Theorem~\ref{prop:Omega1}, there exists a simple closed curve $\Omega'$ contained in the union (\ref{eq:Omegaprime}) with intersections of $H^\pm$ with $\{ u = u_\mp \}$ given by points $z^\pm$.
    We show this curve perturbs to the simple closed curve $\Omega_2$ for $\varepsilon >0$ sufficiently small again using Fenichel's Theorem.
    In the limit $\varepsilon = 0$ the forward orbit from $z^-$ connects to $p_+ \in H^+$ in $\{ u = u_+ \}$.
    The base point $p_+$ is a saddle for all values of $\varepsilon$ and $\{ u = u_+ \}$ thus perturbs to the $O(\varepsilon)$-close stable manifold $\Wstab(p_+,\phitil^{a_+})$ on $U \times N$.
    The graph $H^+$ is uniformly normally hyperbolic and again perturbs to the graph $H^+_\varepsilon$.
    Because $A$ is hyperbolic, there are no equilibria for $\phitil^{a_+}$ on the vertical strip $M^+ = (u_+,w) \times N$, where\footnote{$w<\infty$ by \ref{H-dissipative}.} $w = \inf\{ u > u_- \mid h'(u,a_+) = 0\}$.
    The intersection $M^+ \cap H^+_\varepsilon$ contains a single trajectory for $\phitil^{a_+}$, which by the same argument as in Theorem~\ref{prop:Omega1} is strictly increasing in $u$.
    It follows that $H^+_\varepsilon \cap M = \Wunst(p_+,\phitil^{a_+}) \cap M$ is the unique unstable manifold connecting $p^+$ to the adjacent stable node.
    The segment $\{ u = u_+ \} \cup H^+$ for $\Omega'$ in (\ref{eq:Omegaprime}) thus perturbs to $\Wstab(p_+,\phitil^{a_+}) \cup \Wunst (p_+,\phitil^{a_+})$.
    Using boundedness of $|h'_\star(u)|$ we can again ensure that stable and unstable manifold remain trapped on opposite sides of\footnote{This holds locally around $p_+$ by tangency of the stable and unstable manifolds to the linear stable and unstable subspaces of the linearized flow $D\phitil^{a_+}$ (see (\ref{eq:Jac})) and $h_\star'(u_+)>0$ by assumption. } $H^\star$ on $M^+$.
    The case for the saddle $p_-$ is analogous, interchanging $a_+$ and $a_-$ we can show the segment contained in $\{ u = u_- \} \cup H^-$ perturbs to $\Wstab(p_-,\phitil^{a_-}) \cup \Wunst (p_-,\phitil^{a_-})$.
    The intersections $H^\pm \cap \{ u = u_\mp \}$ are transversal at points $z^\pm$ and persist for $\varepsilon>0$ sufficiently small as transversal intersections at $O(\varepsilon)$-close points $z^\pm_\varepsilon \in \Wunst(p_\pm,\phitil^{a_\pm}) \cap \Wstab(p_\mp,\phitil^{a_\mp})$.
    The resulting curve $\Omega_2$ is depicted in Figure~\ref{fig:omega2}(a).
    \\
    \\
    (ii) The orientation of $\Omega_2$ with respect to $\phitil^{a_\pm}$ is again counter-clockwise in forward time.
    From Lemma~\ref{prop:angle_condition} for $x \in \Wunst(p_\pm,\phitil^{a_\pm})$ we still have $\inner{\ftil(x,a_\pm)^\perp}{\ftil(x,a)} \geq 0$ for all $a \in A$, such that $\ftil(x,A)$ consists of vectors either tangent to $\Omega_2$ or pointing into $\enc \Omega_2$ (cp. Theorem~\ref{prop:Omega1}). However, for $x \in \Wstab(p_\pm,\phitil^{a_\pm})$ now the opposite holds.
    To show $\enc \Omega_2$ is controllable, we fix any pair $x_0,x_1 \in \enc \Omega_2$.
    The backward orbit $\gamma_-$ (\ref{eq:gamma-}), must have a transversal intersection with $\Wunst(p_-,\phitil^{a_-})$ in a point $y_2$.
    The forward orbit $\gamma(x_0,\phitil^{a_+})$ must have a transversal intersection with $\Wstab(p_-,\phitil^{a_-})$ in a point $y_1$.
    For any neighborhood $V_2 \subset \gamma_-$ containing $y_2$ there exists a neighborhood $V_1 \subset \gamma(x_0,\phitil^{a_+})$ mapped into $V_2$ under $\phitil^{a_-}$ by continuity of the flow $\phitil^{a_-}$, passing arbitrarily close to the saddle $p_-$.
    The result then follows and for $\Omega_2 \cap E = \emptyset$ again transfers to the full system using Lemma~\ref{prop:equivalence}.
\end{proof}

\begin{remark}\label{rem:noreturn}
    For $\enc \Omega_2$ the forward invariance property is replaced by the weaker property that, once exited, a solution cannot be steered back into $\enc \Omega_2$ (no-return condition \cite{gayer2004}, see also Appendix~\ref{app:bendixson}).
    In our case, solutions may enter through unstable manifold segments and leave only through stable manifold segments of $\Omega_2$.
    We also remark that a similar construction was used previously in a chemical reaction model where $\enc \Omega_2$ can be foliated directly by the stable and unstable manifolds of the saddles along $\IA = \Punst_A$ (see \cite{ColoniusKliemann2000}, Theorem 9.1.1).
\end{remark}

Lastly, we consider the case (3.) as a Corollary of Theorems~\ref{prop:Omega1} and \ref{prop:Omega2}.
Here, both endpoints must be equilibria for either $\phitil^{a_-}$ or $\phitil^{a_+}$ and we must distinguish between two possibilities.

\begin{manualtheorem}{1.3}\label{prop:Omega3}
    Let $A=[a_-,a_+]$ be hyperbolic and let $I_{\mathcal{P}_A}$ be a connected component with endpoints $(p,a_-) \in \Pstab_A$ and $(q,a_-) \in \Punst_A$. For $\varepsilon>0$ sufficiently small there exists a simple closed curve
    \begin{equation*}
        \Omega_3 \subset \Big[ \gamma(p,\phitil^{a_+}) \cup \Wstab(q,\phitil^{a_-}) \cup \Wunst(q,\phitil^{a_-}) \Big],
    \end{equation*}
    such that $\enc \Omega_3$ is controllable for $\widetilde{F}$.
    If instead $\IA$ has endpoints $(p,a_+) \in \Pstab_A$ and $(q,a_+) \in \Punst_A$, the same holds interchanging $a_-$ and $a_+$.
    If $\Omega_3 \cap E = \emptyset$, then it is also controllable for $F$.
\end{manualtheorem}

\begin{proof}
    Denote $p=(u_1,n_1)^T$ and $q=(u_2,p_2)^T$.
    $\IA$ is transversal to $\{a = a_\pm\}$ in the opposite directions $a'_\star(u_1) > 0$ and $a'_\star(u_1) < 0$, so $u_1 < u_2$.
    We again perturb the simple closed curve $\Omega'$ in (\ref{eq:Omegaprime}) for small $\varepsilon>0$.
    Arguing as for Theorem~\ref{prop:Omega2}, the segment $\{ u = u_2\} \cup H^-$ perturbs to $\Wstab(q,\phitil^{a_-}) \cup \Wunst(q,\phitil^{a_-})$.
    From $A$ hyperbolic there exist no equilibria for the flow $\phitil^{a_+}$ on the set $M = (v,w) \times N$ where $w = \sup\{u<u_1 \mid h'(u,a_+) = 0\}$ and $w = \inf\{u>u_2 \mid h'(u,a_+)=0\}$.
    Arguing as for Theorem~\ref{prop:Omega1}, the solution $\phitil^{a_+}(p)$ exponentially tracks a monotone solution in $H^+_\varepsilon$, intersecting $\Wstab(q,\phitil^{a_-})$. 
    Controllability follows from the procedure used in the proof of Theorem~\ref{prop:Omega1}, interchanging $a_+$ and $a_-$.
\end{proof}

\begin{figure}[ht]
 \begin{center}
  \includegraphics[width=1\textwidth]{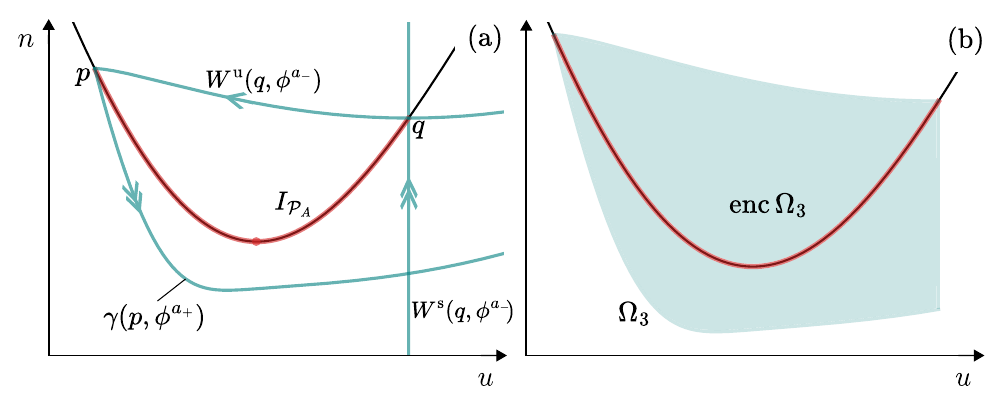}
  \caption{Example of a composition type set $\enc \Omega_3$ surrounding a connected component with endpoints $(p,a_-) \in \Pstab_A$ and $(q,a_-) \in \Punst_A$. (a) Construction of the boundary $\Omega_3$ within the union of stable and unstable manifolds of the endpoint $q$ and the forward orbit $\gamma(p,\phi^{a_+})$ leading to an adjacent stable node. The set does not have the forward invariance property as solutions can escape through the right vertical boundary. (b) The set enclosed by $\Omega_3$ with its controllability property. Sets of this type will result from a collision of a type (i) and (ii) set; see Remark~\ref{rem:bif}. }
  \label{fig:omega3}
  \end{center}
 \end{figure}

In Theorem~1 we have explicitly constructed controllable subsets.
However, the aim of containing the disease trajectory indefinitely does not generally require arbitrary, precise adjustments to the system state.
Instead of Definition~\ref{def:control}, an alternative weaker notion for a subset $K \subset \mathbb{R}^d$ would be the following:
For any initial state $x_0 \in K$ there exists an input $\alpha \in \mathcal{A}$ such that the trajectory remains confined to $K$ for all forward time.
This is known as weak forward invariance \cite{hautus1977}, or in the case where $K$ is closed, viability subsets as extensively studied by Aubin \cite{aubin1984}.
For cancer modelling, such sets were studied numerically recently as stability kernels \cite{Carrere2019} and safe sets \cite{ahmadi2024}.
Candidates for safe subsets are limited for planar differential inclusions in general, where in direct analogy to the Poincar\'{e}-Bendixson Theorem, $K$ can be disjoint from the manifold of feasible equilibria $H^\star\lvert_{a\in A}$ only if it is annular \cite{Filippov1988,hautus1977}.
In particular, for systems of the form (\ref{sys:full}), where one component evolves as the gradient of the other, no annular weakly invariant sets exist and the long-time behavior of trajectories is determined entirely by controllable sets (see Theorem~\ref{prop:limitsets}).
Solutions for any $\alpha \in \mathcal{A}$ escape from a set $K \cap \overline{\enc \Omega} = \emptyset$ in finite time.
Solutions on the complement of all controllable sets cannot self-intersect or accumulate (see Appendix~\ref{app:bendixson}).

\begin{manualtheorem}{2}\label{prop:limitsets}
    Let $A = [a_-,a_+]$ be hyperbolic and $\varepsilon > 0$ sufficiently small.
    Then a solution of (\ref{sys:full}) from $x \in X \setminus S$, for any input $\alpha \in \mathcal{A}$, limits to a set $\overline{\enc \Omega_j}$ of type $j=1,2,3$.
\end{manualtheorem}

For a precise restatement of Theorem~\ref{prop:limitsets} and its proof, see Appendix~\ref{app:bendixson}.
We conclude this Section with two Remarks on the case where $A$ is not hyperbolic and the case where $\alpha$ is not piecewise constant.

\begin{remark}\label{rem:bif}
    If we consider the parameter dependent input range $A^\delta = [a_- - \delta,a_+ + \delta]$, for $\delta \in (-\delta_0,\delta_0)$
    the number and type of controllable subsets may change with the parameter.
    For $A^0$ hyperbolic, no such change can occur for sufficiently small $\delta$ by structural stability in neighborhoods of the endpoints of $I_{\mathcal{P}_{A^0}}$ (see also Appendix~\ref{app:bendixson}). 
    Sufficiently large values of $\delta$ may result in collisions of two or more curves $\Omega_j$ (see also Figure~\ref{fig:bifurcation}).
    While the minimal forward invariance property can be destroyed in such a collision (cp. \cite{homburg2010,lamb2015}), the enclosed sets merge beyond the bifurcation point, retaining controllability.
    For more general continuity properties and collisions for controllable sets on manifolds we refer to \cite{gayer2004}.
\end{remark}

\begin{remark}\label{rem:measurable}
    Considering, instead of the class of piecewise constant inputs $\mathcal{A}$, the larger class of measurable input functions, denoted $\overline{\mathcal{A}}$, Theorem 1 remains true in at least the approximate sense. Closures of reachable sets for both input classes coincide (\cite[Lemma 2.8.2]{sontag2013}). For controllability this means it is possible to reach any open neighborhood of the target point (see also \cite{ColoniusKliemann2000}).
\end{remark}

\begin{figure}[ht]
 \begin{center}
  \includegraphics[width=1\textwidth]{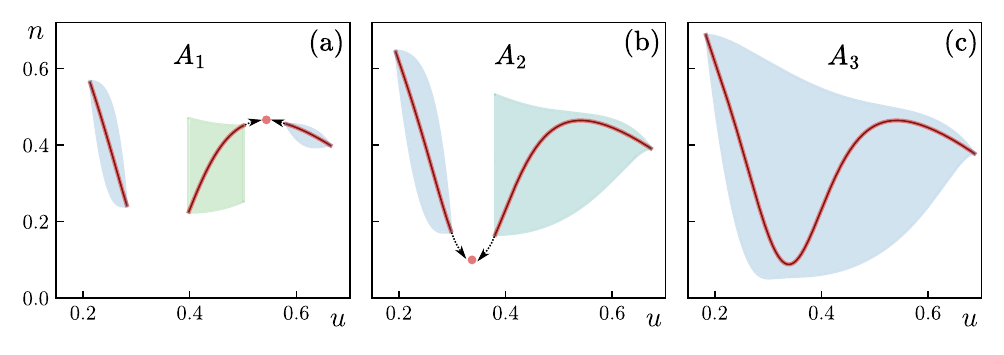}
  \caption{Example of a bifurcation for the growth landscape in example Fig~\ref{fig:potentials}(c) to illustrate Remark~\ref{rem:bif}. Three treatment ranges ordered as $A_1 \subset A_2 \subset A_3$ and the corresponding controllable sets. (a) Two node-type sets described in Theorem~\ref{prop:Omega1} and one of saddle-type described in Theorem~\ref{prop:Omega2}. (b) A collision of a node- and saddle-type set has formed a set of half-stable type (Theorem~\ref{prop:Omega3}, case (ii)). (c) A further collision returns a large node-type set, with forward invariance property restored. }
  \label{fig:bifurcation}
  \end{center}
 \end{figure}

\subsection{Examples}\label{sec:examples2}

We apply our results to the examples introduced in Section~\ref{sec:examples1}.
For examples (a) and (b) we find in both cases a single node-type controllable set $\enc \Omega_1$, globally forward attracting on $X\setminus S$ in the sense of Theorem~\ref{prop:limitsets}.
If a solution hits $\Omega_1$ in finite time, it remains trapped in $\overline{\enc \Omega_1}$ forever.
In the complement of the controllable set, trajectories for all possible inputs $\alpha$ are very limited and given by open curves. 
In example (b) this is predominant, with the patient state largely unresponsive to the drug and only a negligible controllable region (see also Figure~\ref{fig:palliative}).
In example (c) for a input range $A_1$ we find two node-type sets, separated by a saddle-type set $\enc \Omega_2$; see Figure~\ref{fig:bifurcation}.
Scaling the input range as described in Remark~\ref{rem:bif} the sets grow along the connected components and collide, merging first into a set $\enc \Omega_3$ and  eventually into a globally attracting node type set.
Collisions coincide with classical fold bifurcations points of $\phitil^{a_\pm \pm \delta}$.

\begin{figure}[ht]
 \begin{center}
  \includegraphics[width=0.96\textwidth]{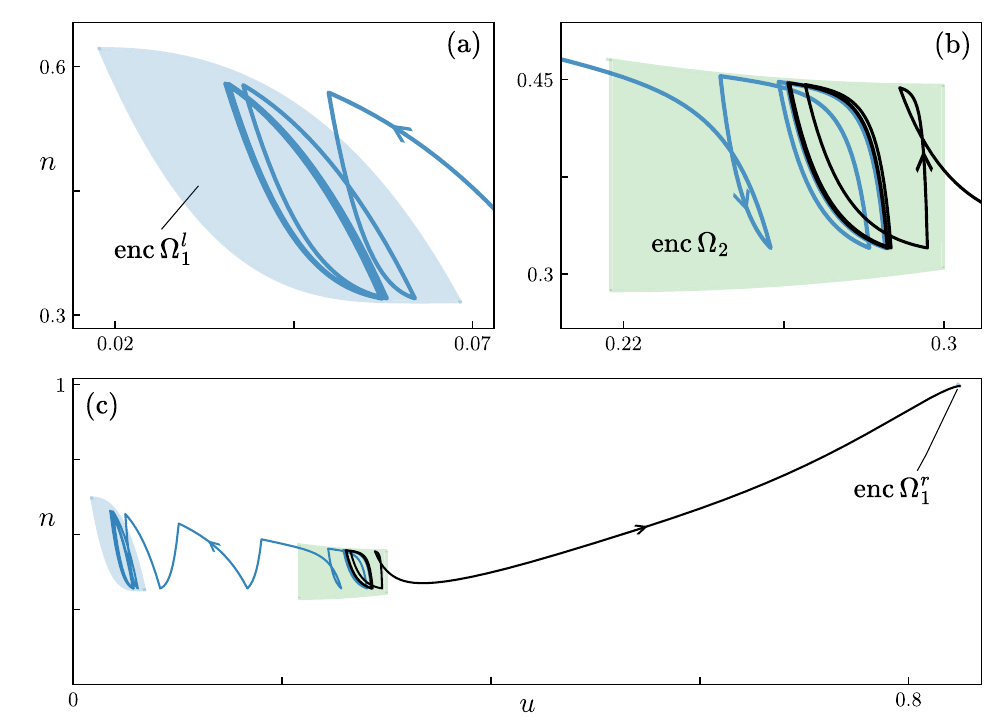}
  \caption{Optimal selections in the landscape $h(\cdot,a)$ considered in example Fig.\ref{fig:potentials}(d).
  The trajectory for initial condition $x_0 = (0.28585,0.32)$ exits right (black curve) and optimization fails, after which treatment is switched to the maximal dose $a = a_M$.  Optimal trajectory for the close-by initial condition $x_0 = (0.28587,0.32)$ instead exits left without failing, transitioning into $\enc \Omega^l_1$. }
  \label{fig:palliative}
  \end{center}
 \end{figure}

Increasing the bounds of the control range, i.e. the instantaneous applied dosage, always increases the size of the controllable sets, though the potential benefits of the forward invariance property can be lost in the process.
Example (d) captures locally the behavior discussed for cases (a) and (b); see Figure~\ref{fig:palliative}.
We have seen above that candidate subsets $K \subset X$ for indefinite containment must have nonempty intersections with controllable subsets. 
In fact, candidates might be further restricted by selecting only for those inputs that are optimal, in some sense, with respect to a performance index (called ``quality of life'') for accumulating toxic side-effects of the therapy.
As an example we consider the landscape in case (d) and choose as the performance index the $L_1$-norm of the treatment over an interval of fixed length $[0,T]$ with a periodicity constraint on the total cell number:
\begin{equation}
        \min_{\alpha \in \overline{\mathcal{A}}} \; \int_{[0,T]} \alpha(t) \, dt \; \text{ s.t. } n(0) = n(T). \label{eq:optimalcontrol}
\end{equation}
We solve the optimal control problem using a Forward-Backward Sweep method, combined with a shooting method for the constraint on the cell number; see \cite[Chapter 21]{lenhart2007}.
Successfully iterating the problem yields a periodically cycling tumor burden, optimal over each period and for which the cell number after the treatment period returns to its initial value.
For an initial condition $x_0 \in \enc \Omega_2$, the solution for the generated input will typically escape from the controllable set and transition towards the left- or right controllable set. 
These transitions are irreversible, see Remark~\ref{rem:noreturn}, and the $u$-components of the trajectories are monotone on the complement of the controllable sets in this example.
In the left transition, the solution eventually settles in $\enc \Omega^l_1$ cycling for all forward time.
A cycle here consists an interval in which the resistance increases, followed by an interval where the drug is removed to re-sensitize the cells and maintain efficacy.
Our choice of performance index does not penalize excursions of the patient trajectory into areas of large $n$, as long as it returns to the initial population periodically.
In the right transition, the objective fails and the trajectory moves towards states with high resistance, unresponsive to the further treatment. 
We know the right transition can be prevented by using controllability and instead selecting at points $x \in \enc \Omega_2$ close to the boundary only for those inputs that drive the trajectory to the left (a feedback control $\alpha_0: X \rightarrow A$).
The definition of quality of life is motivated by the specifics of the cancer model studied and many different definitions exist,  motivated often at least in part by numerical convenience \cite{ledzewicz2022}.
Instead of defining it directly as a functional to be optimized, it can also be implemented by defining an auxiliary dynamical system for the toxicity and drug concentration (see \cite{ahmadi2024}). 
A given choice might disqualify certain controllable subsets as suitable candidates for containment.
However, if for an initial value $x_0 \in X \setminus S$ an admissible input $\alpha_0 \in \overline{\mathcal{A}}$ can found, as an open-loop control $\alpha_0(t)$, feedback control $\alpha_0(x(t))$ or based on other criteria, the asymptotic behavior of the trajectory will be determined by a controllable set.
Since these must enclose subsets of the manifold of equilibria, the idea of containment by only ``sporadic'' application of treatment (cp. \cite{west2023}) resulting in a trajectory consistently traversing parts of $\enc \Omega$, should be contrasted with the idea of stabilizing a specific equilibrium $(p,a) \in \IA$ without treatment breaks (cp. \cite{cunningham2020}).

\begin{figure}[ht]
 \begin{center}
  \includegraphics[width=0.9\textwidth]{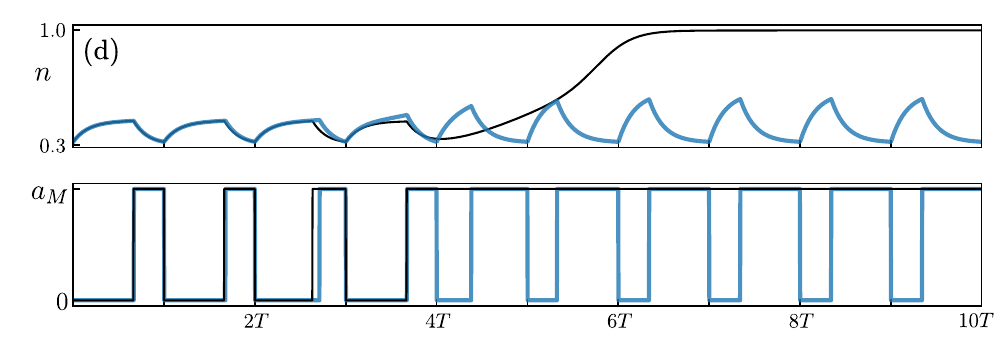}
  \caption{The population numbers $n(t)$ and the optimal selection $\alpha(t)$ in time over 10 treatment intervals of length $T$.}
  \label{fig:control}
  \end{center}
 \end{figure}

\section{Concluding remarks}\label{sec:conclusion}

In this work, we determined boundaries of (maximal) controllable sets from orbits of autonomous vector fields at extreme values of the input range.
These boundaries enclose connected components of the manifold of equilibria in the state space.
We find that they entirely determine the long-time behavior of the system.
Our proof of controllability and invariance properties relies on the constituent orbits being trapped on opposite sides of the manifold or equilibria, where the relative orientation of the family $F$ of vector fields is fixed.
The requirement that the manifolds $H^\pm$ be normally hyperbolic is mostly technical to ensure the curves are closed.
Using smallness of the rate parameter, they can be constructed first in the limit $\varepsilon=0$ and then simply perturbed to form the relevant set for $\varepsilon>0$. 
In most examples, we find that $A$ being hyperbolic is enough to ensure closedness. In that case $\varepsilon$ must still be small enough to ``untangle'' the orbits by ensuring the extremal vector fields are locally transversal to the graph $H^\star$, but this condition is typically slightly weaker.
The method can also be used when the autonomous attractors for the family $ \phi^a $ are more complex, for example including periodic orbits.
We also expect it can be extended to sets of equilbria $\mathcal{P}$ not given globally by a graph or more complex bifurcations diagrams with multiple intersecting branches. 
The significance of extreme points of the input range for the boundaries of reachable sets is well established in geometric control theory. 
The case where the image of $f(\cdot,A)$ is a linear segment, as is the case here, was studied for example in \cite{baytman1971,schattler1996}.
The case where it is instead a set with full dimension was studied more recently in the context of random dynamical systems with bounded noise \cite{homburg2010}.
Here, the input range $A \subset \mathbb{R}^2$ is the unit disc and the set of equilibria then also has full dimension.
Their methods can be used to extend our results to planar models of multi-drug therapy (see e.g. \cite{ma2021}) or to include bounded uncertainties in the system parameters. 
In this work, we considered specifically the idea of indefinite containment for initial conditions from which cure is impossible under any feasible treatment input.
Distinguishing indefinite containment involving controllable sets from ``finite but life-long'' containment on sets disjoint from controllable sets requires more precise knowledge of the involved parameter values and timescales, which we lack here.
In contrast to models for which the evolutionary rate function $k$ is a constant (e.g. \cite{Traulsen2023}), in our model class the curative set (\ref{eq:S}) always has nonempty interior, containing at least stable manifolds of normally attracting subsets of the set $E$ \eqref{eq:E}.
Determining the extent to which any time-varying treatment input yields an improvement in the ability to cure compared to agressive constant-in-time strategies is subject of future work.
We do not expect the particular features we identify for our system to translate well to systems of dimension $d >2$.
However, planar models for cancer therapy are very popular with many recent models sharing the same affine (or even Hamiltonian \cite{gluzman2020}) structure, bounded control range $A \subset \mathbb{R}$ and simple autonomous attractors (e.g. \cite{alvarez2024,greene2018,carrere2017}).

\section*{Acknowledgements}
This research was supported by the European Union’s Horizon 2020 research and innovation programme under the
Marie Skłodowska-Curie EvoGamesPlus grant number 955708.
The authors would like to thank Yannick Viossat, Sebastian Wieczorek, Joel Brown, Kate{\v{r}}ina Sta{\v{n}}kov{\'a} and Jesse Reimann for helpful discussions.

\appendix

\section{Proof of Theorem~\ref{prop:limitsets}}\label{app:bendixson}
We consider system~(\ref{sys:general}) as a differential inclusion, parameterized by $a \in A \subset \mathbb{R}$ as
\begin{equation}\label{sys:inclusion}
    \dot{x} \in \ftil(x,A).
\end{equation}
The set valued map $\ftil(\cdot,A)$ is continuous \cite{hautus1977} with compact, convex image given by a linear segment.
We denote by $\varphi(t,x_0,\alpha)$ the unique, absolutely continuous solution of (\ref{sys:inclusion}) for the initial value $x(t_0)=x_0$ always with $t_0=0$ under the input $\alpha \in \mathcal{A}$. 
We denote the set of feasible equilibria as $H^\star\rvert_{a\in A} = H^\star_A \subset U \times N$.
We use the following notion of an $\omega$-limit set for nonautonomous dynamical systems (\cite{Filippov1988}, Section 12):
\begin{equation}
    \omega_\alpha(x_0)= \{ x \in \mathbb{R}^2 \mid \text{ there exists } t_n \rightarrow \infty \text{ such that } \varphi(t_n,x_0,\alpha) \rightarrow x \}. \label{eq:omegalimit}
\end{equation}
This set need not be forward invariant in the sense that for $p \in \omega_\alpha(x_0)$ we generally have $\varphi(t,p,\alpha) \not\in \omega_\alpha(x_0)$.
Crucially, however, for the differential inclusion (\ref{sys:inclusion}) we can compensate by selecting a measurable input $\beta: [0,\infty) \rightarrow A$ such that $\varphi(t,p,\beta) \in \omega_\alpha(x_0)$ for all $t$; see \cite[Lemma 12.4]{Filippov1988}.
Using this property yields a direct generalization of the Poincar\'{e}-Bendixson Theorem (cp. e.g. \cite[Theorem 1.25]{dumortier2006}) to planar differential inclusions:
\begin{manualtheorem}{A}[\cite{Filippov1988} Theorems 13.6 and 13.7, \cite{hautus1977}]\label{prop:hautusPB}
    Suppose the forward orbit $\{ \varphi(t,x_0,\alpha) \mid t \geq 0 \}$ is bounded.
    Then
    \begin{align*}
        \text{ either } \quad (i)& \;\;  \omega_\alpha(x_0) \cap H^\star_A \neq \emptyset, \\
        \text{ or } \quad (ii)& \;\; \omega_\alpha(x_0) \text{ contains a closed orbit } L \text{ disjoint from } H^\star_A \\
                              &\text{ such that } \overline{ \enc L} \cap H^\star_A \neq \emptyset.
    \end{align*}
\end{manualtheorem}

In the following we use the Hausdorff semi-distance $$\dist(X,Y) = \sup_{x \in X} \dist(x,Y) = \sup_{x \in X} \inf_{y \in Y} |x-y|$$ for nonempty subsets $X,Y \subset \mathbb{R}^d$.
We then restate Theorem~\ref{prop:limitsets} as follows

\begin{manualtheorem}{2}
    Let $A = [a_-,a_+]$ be hyperbolic and $\varepsilon>0$ sufficiently small. Then for any $x_0 \in \mathbb{R}^2$ and any $\alpha \in \mathcal{A}$ it holds that
    \begin{equation*}
        \lim_{t \rightarrow \infty} \dist (\varphi(t,x_0,\alpha),\overline{\enc \Omega_j}) = 0,
    \end{equation*}
    for a set $\Omega_j$ of type $j=1,2,3$.
\end{manualtheorem}

\begin{proof}
    We work with slightly enlarged versions of our controllable sets, set up as follows:
    Because $A$ is hyperbolic, there exists $\delta_0 > 0$ sufficiently small such that for all $a \in (a_- - \delta_0, a_+ + \delta_0)$ the vector field $\ftil(\cdot,a)$ is structurally stable in neighborhoods of $p_\pm$.
    We then consider the system (\ref{sys:general}) with parameter-dependent input range $A^\delta = [a_--\delta,a_+ +\delta]$ for $\delta \in (-\delta_0,\delta_0)$.
    For all parameter values, endpoints of connected components $I_{\mathcal{P}_{A^\delta}}$ do not change type and Theorem~1 (adjusting $\varepsilon>0$ if necessary) gives for each component a collection simple closed curves $\Omega^\delta$ surrounding them.
    For $\delta_1<\delta_2$, we have $\Omega^{\delta_1} \cap \Omega^{\delta_2}=\emptyset$ because they are constructed as unions of orbits for different extremal autonomous vector fields $\ftil(x,a_\pm\pm \delta)$ transversal by Lemma~\ref{prop:angle_condition}.
    It follows that the enclosed sets are strictly monotone $\overline{\enc \Omega^{\delta_1}} \subset \overline{\enc \Omega^{\delta_1}}$.
    It then also follows that the set-valued map assigning $\delta \mapsto \overline{\enc \Omega^\delta}$ is upper-semicontinuous on $(-\delta_0,\delta_0)$ (see Scherbina's Lemma \cite[4.1.3]{pilyugin2006}).
    We now consider the two cases described by Theorem~\ref{prop:hautusPB}.
    \\
    \\
    (i) Let $x_0 \in \mathbb{R}^2$ and $\alpha(t) \in A^0$. 
    Assume there exists a point $p \in \omega_\alpha(x_0)\cap I$ for a connected component $I \subset H^\star_A$ contained in a set $\overline{\enc \Omega^0}$.
    For $\delta \in (0,\delta_0)$ by the above we have $p \in \enc \Omega^\delta$.
    The enclosed set is open, so for the sequence $t_n \rightarrow \infty$ there exists $n_\delta \in \mathbb{N}$ such that $\varphi(t_n,x_0,\alpha) \in \enc \Omega^\delta$ for all $n \ge n_\delta$.
    This implies $\varphi(t,x_0,\alpha) \in \overline{\enc \Omega^\delta}$ for all $t \ge t_{n_\delta}$.
    Indeed, suppose $z = \varphi(T,x_0,\alpha) \in \Omega^\delta$ at some $T>t_{n_\delta}$.
    Then we know $z \in \Wstab(q,\phitil^{a_\pm \pm \delta})$ for a saddle point $q$, all other segments of $\Omega^\delta$ being strictly inflowing for $a\in A^0$ (see Remark~\ref{rem:noreturn}).
    Thus, the solution $\varphi(t,x_0,\alpha)$ must for $t>T$ spend some time in the set between the perturbed graphs
    \begin{equation}\label{eq:Bdelta}
        B_\delta = \{ x \in U\times N \mid h_\varepsilon(u,a_+ + \delta)\le n \le h_\varepsilon(u,a_- - \delta)\},
    \end{equation}
    containing $\overline{\enc \Omega^\delta}$. One can verify, again using Lemma~\ref{prop:angle_condition}, that this set is forward invariant in the sense $\varphi(t,B_\delta,\alpha) \subset B_\delta$ for all $\delta \in [0,\delta_0)$.
    Trapped in this way, the solution cannot return to the original controllable set, which gives a contradiction.
    It follows then that $\omega_\alpha(x_0) \subset \overline{\enc \Omega^\delta}$.
    To show that $\omega_\alpha(x_0) \subset \overline{\enc \Omega^0}$ we use that
    \begin{equation*}
        \dist(\omega_\alpha(x_0),\overline{ \enc \Omega^0}) \le \dist (\overline{ \enc \Omega^\delta},\overline{ \enc \Omega^0}) \rightarrow 0 \text{ as } \delta \rightarrow 0,
    \end{equation*}
    where taking the limit we have used upper-semicontinuity in $\delta$.
    Finally, since $\varphi(t,x_0,\alpha)$ is bounded, $\omega_\alpha(x_0)$ is compact and by a standard subsequence argument we then have $\dist(\varphi(t,x_0,\alpha),\omega_\alpha(x_0)) \rightarrow 0$ as $t \rightarrow \infty$. Theorem~\ref{prop:limitsets} holds in case (i).
    \\
    \\
    (ii) Assume $\omega_\alpha(x_0) \cap H^\star_A = \emptyset$.
    Then by Theorem~\ref{prop:hautusPB} there exists a closed trajectory $L = \{\varphi(t,y,\beta) \mid t \in \mathbb{R}\}$ contained in $\omega_\alpha(x_0)$ such that $I \subset \enc L$ for at least one connected component $I \subset H^\star_A$.
    There must then exist a point $q \in L \cap (\mathbb{R}^2 \setminus B_0)$ (see \eqref{eq:Bdelta}) such that $q = \varphi(t_1,y,\beta)=\varphi(t_2,y,\beta)$.
    This requires crossing through the forward invariant set $B_0$ and gives a contradiction, so case (ii) can never be satisfied.
\end{proof}

\section{Table of parameters}\label{app:parameters}

We list here the explicit functions and parameters used in the example figures.
For the intrinsic growth function many choices are possible.
We use a multimodal Gaussian distribution across the trait-values, allowing simple adjustment to the position and slope for the extrema
\begin{equation}
    b_0(u) =  \sum_{j}  r_j \exp(-g_j (u - \overline{u}_j)^2). \label{eq:b0example}
\end{equation}
Many other forms can be used as well.
To account for the dissipativity condition \ref{H-dissipative}, we substract from (\ref{eq:b0example}) an even-power polynomial of the form $p_1 (u-p_2)^{2p_3}$ with $p_3 \in \mathbb{N}$.
For the treatment efficacy function we use sigmoids of the form
\begin{equation*}
    b_1(u) = \sum_{j} \frac{c^j_1}{c^j_2 + c^j_3 \exp(c^j_4 (c^j_5 u - c^j_6))} + c^j_7.
\end{equation*}
Other forms, such as the more simple exponentially decaying functions used in \cite{pressley2021} can be used as well.
For simplicity, we assume throughout our examples that $c(u)\equiv 1$ is a constant.

\begin{table}[ht]
    \centering
    \caption{In the main text we consider four types of growth functions (a)--(d) as examples. This table lists the parameters for these examples with the control ranges used in Figures~\ref{fig:potentials},\ref{fig:PA},\ref{fig:bifurcation} and \ref{fig:palliative}.}
    \begin{tabular}{c c c c c}
    \hline 
         & Ex.(a) & Ex.(b) & Ex.(c) & Ex.(d) \\ \hline
        $\mathbf{r}$ & $(0,0.41,0.86)$  & $(0.26,0.4,0.96)$ & $(0.5,0.7,0.35)$ & $(0.6,0.4,0.95)$\\
        $\mathbf{g}$ & $(0,1.9,2.5)$ & $(13,8.9,7.9)$ & $(18.6,9.8,8.8)$ & $(14.3,13.7,13.8)$\\
        $\overline{\mathbf{u}}$ & $(0,0.8,0)$ &$(-0.01,0.35,0.87)$  & $(0,0.25,0.68)$ & $(0,0.47,0.87)$\\
        $\mathbf{p}$ & $(0.1,0.65,3)$ & $(2.8,0.6,6)$ & $(10,0.55,5)$ & $(1,0.46,6)$\\
        $\mathbf{c}^1$ & $(1,0.9,1,0.5,1,0.3,0)$ & $(1,0.9,1,11,1,0.5,0)$ & $(-0.462,1,10.1,$ & $(1,0.9,1,6.7,$\\
        &  &  & $-1.44,10,0,0)$ & $1,0.2,1,0)$ \\
        $\mathbf{c}^2$ & -- & -- & $(-0.633,1,10.1$ & -- \\
         & & & $-1.44,10,3.7,1.1)$ & \\ \hline
        Fig.\ref{fig:potentials} & $A=[0,1.75]$ & $A=[0,0.7]$ & $A=[0,0.7]$ & $A=[0,1.3]$ \\
        Fig.\ref{fig:PA} & $A=[0.3,1.75]$ & $A=[0,0.7]$ & $A=[0.15,0.8]$ & $A=[0.1,0.45]$ \\
        Fig.\ref{fig:bifurcation} & -- & -- & $A_1 = [0.5,0.95]$ & -- \\
         &  & & $A_2 = [0.4,1.05]$ &  \\
         &  &  & $A_3 = [0.35,1.23]$ &  \\
        Fig.\ref{fig:palliative} & -- & -- & -- & $A=[0,0.38]$ \\ \hline
    \end{tabular} 
    \label{tab:my_label}
\end{table}

\bibliography{references}
\bibliographystyle{amsplain} 

\end{document}